\documentclass[12pt,onecolumn,twoside]{IEEEtran}

\usepackage[utf8]{inputenc}
\usepackage[T1]{fontenc}
\usepackage{url}
\usepackage{ifthen}
\usepackage{cite}
\usepackage[cmex10]{amsmath}
\usepackage[affil-it]{authblk}
\usepackage[usenames,dvipsnames]{xcolor}
\usepackage{amsfonts}
\usepackage{amsmath,amsthm,amssymb,dsfont}
\usepackage{enumerate}
\usepackage[english]{babel}
\usepackage{graphicx}	
\usepackage[caption=false]{subfig}
\usepackage{url}
\usepackage{todonotes}
\usepackage{bbm}
\usepackage{booktabs}
\usepackage{stmaryrd}
\usepackage{tikz}
\usetikzlibrary{chains}
\usetikzlibrary{fit}
\usepackage{pgflibraryarrows}		
\usepackage{pgflibrarysnakes}		
\usepackage{caption}

\usepackage{makecell}

\usepackage{diagbox}

\usepackage{epsfig}
\usetikzlibrary{shapes.symbols,patterns} 
\usepackage{pgfplots}
\pgfplotsset{compat=newest}

\usepackage[linesnumbered,ruled,vlined]{algorithm2e}
\usepackage{algorithmic}
\usepackage{hyperref}
\hypersetup{colorlinks=true,citecolor=blue,linkcolor=blue,filecolor=blue,urlcolor=blue,breaklinks=true}

\usepackage{nicefrac}
\usepackage{mathtools}
\usepackage{multicol}
\usepackage{mathrsfs}
\usepackage{multirow}
\usepackage{bm}

\theoremstyle{plain}
\newtheorem{theorem}{Theorem}
\newtheorem{lemma}[theorem]{Lemma}

\newtheorem{corollary}[theorem]{Corollary}
\newtheorem{proposition}[theorem]{Proposition}

\theoremstyle{definition}

\newtheorem{remark}[theorem]{Remark}
\newtheorem{example}[theorem]{Example}
\newtheorem{construction}{Construction}

\newcommand*{\cA}{\mathcal{A}}
\newcommand*{\cB}{\mathcal{B}}
\newcommand*{\cC}{\mathcal{C}}

\newcommand*{\cE}{\mathcal{E}}

\newcommand*{\cL}{\mathcal{L}}

\newcommand*{\cR}{\mathcal{R}}
\newcommand*{\cS}{\mathcal{S}}
\newcommand*{\cT}{\mathcal{T}}

\newcommand*{\cW}{\mathcal{W}}
\newcommand*{\cX}{\mathcal{X}}
\newcommand*{\cY}{\mathcal{Y}}

\newcommand*{\bigO}{\mathcal{O}}

\newcommand{\abs}[1]{\left|#1\right|}

\newcommand{\field}[1]{\mathbb{#1}}

\newcommand{\F}{\field{F}}

\newcommand{\Z}{\field{Z}}
\newcommand{\dS}{\field{S}}

\newcommand{\suppress}[1]{}

\newcommand{\blda}{{\mathbf{a}}}
\newcommand{\bldb}{{\mathbf{b}}}
\newcommand{\mmod}{{\mbox{mod}}}

\SetKwInput{KwInput}{Input}                
\SetKwInput{KwOutput}{Output}              

\begin{document}

\title{Constructions of Covering Sequences \\ and 2D-sequences}

\author{Yeow Meng Chee\IEEEauthorrefmark{1},
        Tuvi Etzion\IEEEauthorrefmark{2},
        Hoang Ta\IEEEauthorrefmark{3},
        and Van Khu Vu\IEEEauthorrefmark{1}}

\affil{
\IEEEauthorrefmark{1}\small{Singapore University of Technology and Design, Singapore}\\
\IEEEauthorrefmark{2}\small{Department of Computer Science, Technion --- Israel Institute of Technology, Haifa, 3200003 Israel}\\
\IEEEauthorrefmark{3}\small{Hanoi University of Science and Technology, Vietnam}
}

\maketitle

\begin{abstract}
An $(n,R)$-covering sequence is a cyclic sequence whose consecutive $n$-tuples form a code of
length $n$ and covering radius $R$. Using several construction methods improvements
of the upper bounds on the length of such sequences for $n \leq 20$ and $1 \leq R \leq 3$,
are obtained. The definition is generalized in two directions.
An $(n,m,R)$-covering sequence code is a set of cyclic sequences of length $m$ whose consecutive
$n$-tuples form a code of length~$n$ and covering radius $R$.
The definition is also generalized to arrays in which the $m \times n$ sub-matrices form a covering
code with covering radius $R$. We prove that asymptotically there are covering sequences that attain the sphere-covering
bound up to a constant factor.
\end{abstract}

\section{Introduction}
\label{sec:introduction}
An {\bf \emph{$(n,R)$-covering code}} $\cC$ is a set of words of length $n$ over a given alphabet $\Sigma_q$, of size $q$, such that each word
of length $n$ over $\Sigma_q$ is within distance $R$ from at least one codeword in $\cC$.
In other words, for each $x \in \Sigma_q^n$, there exists $c \in \cC$ such that $d(x,c) \leq R$,
where $d(y,z)$, $y,z \in \Sigma_q^n$, denotes the Hamming distance between $y$ and $z$.
Covering code were always of interest, but the interest increased due to the following three seminal papers~\cite{CKMS,CLS86,GrSl85}.
The interest was also increased partially because of the connection of covering codes to data compression.
An excellent book that covers all aspects of such codes is~\cite{CHLL97}.

Chung and Cooper~\cite{ChCo04} generalized the notion of an $(n,R)$-covering code of length $n$ and radius $R$
to a cyclic sequence whose consecutive $n$-tuples form a covering code of length $n$ and radius $R$.
They called such a structure a de Bruijn covering code since $n$-tuples of a cyclic sequence are considered
and this sequence forms a cycle in the de Bruijn graph.
However, any cyclic sequence can be viewed as a cycle in the de Bruijn graph, and the sequence itself has a very
loose connection to the de Bruijn graph, although some of the constructions in the paper
use ingredients and concepts from the de Bruijn graph. Therefore, we prefer to call such a sequence a {\bf \emph{covering sequence}}.
This is a dual definition of what is known as a robust sequence, i.e., a sequence
whose consecutive $n$-tuples form an error-correcting code of length $n$ and minimum distance $d$. Such sequences as well as arrays have been considered
for example in~\cite{BeKo16,CDKLW20,BEGGHS,HoSt16,KuWe92,MiWi22,WWO17,Wei22} and they have variety of applications.

All the discussion from Section~\ref{sec:cyclic} will be restricted to the binary alphabet and hence alphabet size will rarely be mentioned.
However, many of the ideas that will
be presented can be generalized quite easily to any alphabet size.
An {\bf \emph{$(n,R)$-covering sequence}} (an $(n,R)$-CS for short) is a cyclic sequence whose consecutive $n$-tuples
form an $(n,R)$-covering code. This structure was first defined and analyzed in~\cite{ChCo04}.
To find the shortest $(n,R)$-CS, denoted by $\cL(n,R)$, we define two more related structures of cyclic sequence codes.
A cyclic covering sequence code contains sequences in which the consecutive $n$-tuples of all the sequences
form an $(n,R)$-covering code. Such codes will be discussed in our exposition.

In recent years many one-dimensional coding problems have been considered in the two-dimensional framework due to
modern applications. This is quite natural and has become fashionable from
both theoretical and practical points of view. Such generalizations were considered
for various structures such as error-correcting codes~\cite{BlBr00,Rot91}, burst-correcting codes~\cite{BBV,BBZS,EtVa02,EtYa09},
constrained codes~\cite{ScBr08,TER09}, de Bruijn sequences~\cite{Etz88,Etz24b,Mit95,Pat94}, M-sequences
where the two dimensional sequences are pseudo-random arrays~\cite{Etz24a,McSl76}. Other structures designed for applications are
robust self-location arrays with window property~\cite{BEGGHS} and
structured-light patterns~\cite{MOCDZN}.
Therefore, it is very tempting to generalize the concept of covering sequences into a two-dimensional framework
and this is one of the targets of the current work.

An {\bf \emph{$(m \times n,R)$-covering 2D-sequence}} (an $(m \times n,R)$-C2DS for short) is a doubly-periodic $M \times N$ array
(cyclic horizontally and vertically like a torus)
over an alphabet of size $q$ such that the set
of all its $m \times n$ windows form a covering code with radius~$R$.

While in the one-dimensional
case we are interested in the $(n,R)$-CS of the shortest length, in the two-dimensional
case we are interested in the $(m \times n,R)$-C2DS with the smallest area, but the ratio between $M$ and $N$ can be important too.

\begin{remark}
It is tempting to use the term "covering array" for the two-dimensional matrices, but this term is
already reserved to another combinatorial object associated with covering, see~\cite{Col04,CKSS,Slo93} and references therein.
\end{remark}

Our goals in this paper are to present construction methods for covering sequences, covering sequence codes, and covering 2D-sequences.
Some of the constructions yield sequences and codes that are almost optimal.

The rest of the paper is organized as follows.
In Section~\ref{sec:prelim}, some basic results and important known results on covering codes
and covering sequences are presented.
Section~\ref{sec:cyclic} is devoted to cyclic covering sequence codes and the constructions of covering sequences from cyclic covering sequence codes.
In Section~\ref{sec:self-dual} a cyclic covering sequence code based on self-dual sequences is presented.
Based on this code a relatively short $(2^k,1)$-CS whose
length is within factor of 1.25 from the sphere-covering bound is obtained. An interleaving construction to obtain an $(n,R)$-CS
from an $(n_1,R_1)$-CS and an $(n_2,R_2)$-CS, where $n=n_1+n_2$ and $R=R_1+R_2$ is presented in Section~\ref{sec:interleave}.
When $n_1=n_2$ and $R_1=R_2$ a better interleaving construction is presented.
A construction based on primitive polynomials, with a certain structure, is presented in Section~\ref{sec:primitive}.
Section~\ref{sec:folding} is the only one devoted to covering 2D-sequences and introduces two methods to generate such arrays:
one is by folding a covering sequence and the other is by making all possible shifts of a related covering sequence.
These methods are used to obtain upper bounds on the size of such arrays and to construct them.
Finally, conclusion and further problems for future research are presented in Section~\ref{sec:conclude}.

\section{Preliminaries}
\label{sec:prelim}

The area of covering codes is very well established in information theory as a dual for error-correcting codes.
Bounds on the sizes of such codes were extensively studied, where the upper bounds are either by constructions or using probabilistic
methods to prove the existence of some codes asymptotically. Lower bounds are usually obtained by analytic methods, but they
are usually not much better than the sphere-covering bound which is the most basic bound.

A ball of radius $R$ around a word $\bm{x}$ of length $n$ is the set of words whose distance is at most $R$ from~$\bm{x}$.
The size of such ball denoted by $V_q(n,R)$ is
$$
V_q(n,R) = \sum_{i=0}^{R} \binom{n}{i} (q-1)^i ~.
$$
Since the balls of radius $R$ around the codewords of an $(n,R)$-covering code $\cC$ contain the
whole space, it follows that a lower bound on the size of $\cC$ is
$$
\abs{\cC} \geq \frac{q^n}{V_q (n,R)}.
$$
This bound is the {\bf \emph{sphere-covering bound}}.
There exists a covering code that approaches this bound up to a factor roughly $eR \log R$~\cite{KSV03}.
The proof method for this bound is probabilistic.
A similar bound for an $(n,R)$-CS over a prime power alphabet was presented in~\cite{ChCo04}. This bound was generalized to any alphabet by
Vu~\cite{vu05}. The bound states that for fixed $R$
there exists an $(n,R)$-CS, over $\Sigma_q$, whose length is at most $\bigO\left(\frac{q^n}{V_q(n,R)}\log n \right)$.

This result of Vu~\cite{vu05} was generalized in our conference paper~\cite{CETV24} where we proved the following
\begin{proposition}
\label{prop:prob_bound}
Let $m,n$ be nonnegative integers. For any $M\geq m$, there exists an $M \times N$ $(m \times n,R)$-C2DS
such that  $M \cdot N = \bigO \left( \frac{q^{mn}}{V_q(mn,R)} \cdot (\log m + \log n)  \right)$,
for fixed $q$ and $R$.
\end{proposition}
The proof of Proposition~\ref{prop:prob_bound} was presented in~\cite{CETV24} using carefully the
probabilistic method which was also used in~\cite{vu05} and will not be proved here. The same existence proof can be obtained by applying
folding technique which will be presented in Section~\ref{sec:folding} on the one-dimensional sequences
which are known to exist by the asymptotic bound of~\cite{vu05}.

For small $R$, there are some $(n,R)$-covering codes that attain the sphere-covering bound with equality, such as the Hamming codes
of length $n=2^k-1$ and radius one.
Other codes are very close to the upper bound, such as the ones for $R=2$ which are perfect asymptotically~\cite[Construction 4.24]{Str94}
or other similar codes~\cite{EtGr93,EtMo05}. Similarly, such sparse covering codes were also considered for $R=3$~\cite{EtGr93,EtMo05}.
The main goal of the research on $(n,R)$-CSs is to get as close as possible to the upper bounds obtained for covering codes.

One of our constructions will use a span $n$ de Bruijn sequence over $\Sigma_q$.
This is a cyclic sequence of length $q^n$, where each $n$-tuple over $\Sigma_q$ is contained in exactly one
window of consecutive digits in the sequence. Such sequences exist for all $q \geq 2$ and $n \geq 1$.

Finally, one very trivial bound which was mentioned in~\cite{ChCo04} will be used in one of the best known bounds (see Table~\ref{table:one_dim_2}).
\begin{theorem}
\label{thm:trivialB}
For any $n,R \geq 1$ we have that $\cL(n,R) \leq \cL(n+1,R)$.
\end{theorem}

\section{Construction from a Cyclic Covering Code}
\label{sec:cyclic}

From this section, the discussion will be focused on binary alphabet, but some of the constructions can be adapted to nonbinary alphabet.
This is especially true for this section where all the results can be given also for codes over an arbitrary alphabets.
Cyclic covering codes are perhaps the most important ingredients in constructions of $(n,R)$-CSs, especially when $n$ and $R$ are relatively small.
They were considered in various papers, e.g.~\cite{Dou91,DoSl85,Jan89,KaTu19}.
There are two definitions for covering sequence codes.

An {\bf \emph{$(n,m,R)$-CS code}} (an $(n,m,R)$-CSC for short) $\cC$ is a set of cyclic codewords of length $m$
such that each word of length $n$ is within distance $R$ from at least one $n$-tuple of a codeword of $\cC$, i.e., the
consecutive $n$-tuples of all the codewords form an $(n,R)$-covering code.

An {\bf \emph{$(n,R)$-CS code}} (an $(n,R)$-CSC for short) $\cC$ is a set of cyclic words of possibly different lengths
such that each word of length $n$ is within distance $R$ from at least one window of length $n$ in one of the sequences, i.e., the
consecutive $n$-tuples of all the codewords form an $(n,R)$-covering code.

The only distinction between the two types of codes is that
in the first one of an $(n,m,R)$-CSC, all codewords have the same length $m$ and in the second one of an $(n,R)$-CSC codewords can be of different lengths.
The $(n,R)$-CS is a cyclic sequence and hence it is a cyclic covering sequence code with one codeword.

The natural criteria to compare different $(n,R)$-CSCs is the total length of all the sequences in each code. Another criteria
for comparison between CSCs which is important
in our exposition is the length of the $(n,R)$-CS that can be constructed by combining the codewords of each CSC.
This process is done as follows.

The sequences of the code are combined by first concatenating the first $n-1$ bits of each sequence after its last bit.
This converts a cyclic sequence to an acyclic sequence, where both sequences have the same $n$-tuples. After the sequences become acyclic,
we order the sequences in a cyclic list, where two consecutive sequences in the list overlap such that the suffix of the first sequence
aligns with the prefix of the next sequence. The target is to have the total overlaps (sum of the lengths of all the overlaps) as large
as possible.

For this purpose, when we concatenate the suffix of length $n-1$ after the last bit of the codeword, we should try to do it for
every cyclic shift of the codeword and ultimately use the shifts that yield the largest total overlaps.
Merging all the sequences together is done following the order of the list, where each overlap is taken naturally only once.
Finally, there might be sequences in the code with periodicity. The periodicity should be eliminated either within the code or in the final
$(n,R)$-CS. This is demonstrated in the following example.

\begin{example}
\label{ex:9_1_93}
The following eight codewords form a $(9,10,1)$-CSC $\cC$:
\begin{align*}
& [1000010000], [0001001101], [1001111001], [1111010111], \\
& [1010101010], [0101011000], [0110111001], [0111010000].
\end{align*}

We extend each cyclic sequence by eight bits to obtain all the $9$-tuples of the cyclic sequences in acyclic sequences.
We furthermore find the order of these eight sequences in a way that the total overlap between suffixes and prefixes of
the consecutive sequences (including the last sequence in the list and the first sequence in the list)
will be large as possible, where in this process all possible shifts of all the sequences are considered.
The obtained eight acyclic sequences of length 18 are as follows, where the overlap is written after the sequence:
\begin{align*}
& ~~~the~sequence & overlap \\
& 100001000010000100 &  6~~~~ \\
& 000100110100010011 &  5~~~~ \\
& 100111100110011110 &  5~~~~ \\
& 111101011111110101 &  5~~~~ \\
& 101010101010101010 &  5~~~~ \\
& 010101100001010110 &  4~~~~ \\
& 011011100101101110 &  5~~~~ \\
& 011101000001110100 &  3~~~~
\end{align*}

The total number of bits in the eight sequences is 144 and the total number of overlaps is 38. This yield a $(9,1)$-CS of length $144-38=106$ as follows:
\begin{align*}
& [10000100001000010011010001001111001100111101011111110 \\\
&  10101010101010101100001010110111001011011101000001110]
\end{align*}

However, the codeword $[1000010000]$ of $\cC$ is periodic and contains only five 9-tuples as the cyclic sequence $[10000]$. Moreover, the codeword $[1010101010]$ of $\cC$ is also periodic
and it contains only two 9-tuples as the cyclic sequence $[10]$.
Hence, they are redundant in the code and in the final $(9,1)$-CS. Therefore, we can eliminate the redundancy in the covering sequence and
obtain the following $(9,1)$-CS of length $106-13=93$:
\begin{align*}
& [100001000010011010001001111001100111101011111110 \\
&  101010101100001010110111001011011101000001110]
\end{align*}

Introducing the codewords $[10000]$ instead of $[1000010000]$ and $[10]$
instead of $[1010101010]$ in the code $\cC$ yields a $(9,1)$-CSC with 6 codewords of length 10,
one codeword of length 5, and one codeword of length 2.

This sequence of length 93 can be extended to a $(9,1)$-CS of length 102 with 8 consecutive \emph{ones} which will be required later:
\begin{align*}
& [100001000010011010001001111001100111101011111111011111110 \\
&  101010101100001010110111001011011101000001110]
\end{align*}
\hfill\quad $\blacksquare $
\end{example}

Many of the shortest $(n,R)$-CSs for small $n$ and $R$ were found by computer search. For this purpose we use the well-known
problem of the {\bf \emph{Shortest Cyclic Superstring (SCS)}}. This problem is designed to find the shortest cyclic superstring,
i.e. to find a short sequence as possible which contains all the sequences of a given set $\dS$ as subsequences.
The acyclic version of the SCS problem, which find the shortest superstring, has been extensively studied in the literature.
This problem has numerous applications in data compression~\cite{Sto87} and computational biology~\cite{GMSR}.
However, the problem is known to be ${\text{NP-complete}}$~\cite{GMA80} and even APX-complete~\cite{BJLTY},
indicating that polynomial-time approximation schemes with arbitrary constant approximations are not to be expected.

Many greedy and approximation algorithms have been developed to tackle this problem, and they have demonstrated numerical efficiency~\cite{JiLi94,Ste94,KaSh05}.
Furthermore, we note that any solution for the acyclic SCS problem can also be considered a feasible solution for the cyclic SCS problem.
In our variant to the problem we built a complete directed graph which represent all the sequences of the code $\cC$ and their possible
overlaps with all the other sequences that follow them. For this purpose
we use the well-known approximation algorithm for the set cover problem~\cite[Chapter 2]{Vaz01}.

While comparing codes with codewords of different lengths is not trivial, it is much easier to compare
codes in which all codewords have the same length. Given two $(n,m,R)$-CSCs, it is obvious that the code
with a smaller number of codewords will be considered better.
Given two codes, $\cC_1$ an $(n,m_1,R)$-CSC and $\cC_2$ an $(n,m_2,R)$-CSC where the total length of the codewords
is smaller in $\cC_1$ and $m_1 \geq m_2$, then $\cC_1$ is considered superior to $\cC_2$. However, if the total length of the codewords
in $\cC_1$ is smaller and $m_1 < m_2$, then the two codes are incomparable. Nevertheless, there should be some tradeoff
between the two parameters. We demonstrate these concepts for a length $n=2^k-1$ in this section and for a length $n=2^k$
in Section~\ref{sec:self-dual}.

For a length $n=2^k -1$ we consider the Hamming code $\cC$.
The Hamming code $\cC$ of length $2^k-1$ can be represented as a cyclic covering code that contains $2^{2^k-k-1}$ codewords.
If $c$ is a codeword in $\cC$ then all the cyclic shifts of $c$ are also codewords in $\cC$.
Each codeword of $\cC$ has $d$ distinct cyclic shifts, where $d$ is a divisor of $n$. For example,
the all-zero word is a codeword and it has one distinct cyclic shift. From the set of $d$ distinct cyclic
shifts only one is considered in the code.

\begin{example}
\label{ex:Ham15}
For $n=15$, we consider the cyclic Hamming code of length 15.
It has 134 codewords of length 15, 6 codewords of length 5, 2 codewords of length 3, and 2 codewords of length 1.
Together, these 144 codewords contain 2048 words of length 15 which form the Hamming code of length 15.
They are merged to form a $(15,1)$-CS of length 3516 (see Appendix~\ref{sec:Ham15}) compared to a lower bound 2048 obtained by the sphere-covering bound.
\hfill\quad $\blacksquare $
\end{example}

In general, the Hamming code $\cC$ contains $2^{2^k-k-1}$ codewords. There are at most
$$
\sum_{d|n,~d<n} 2^d < 2^{\lfloor \frac{n}{2} \rfloor +1} = 2^{2^{k-1}}
$$
cyclic words that contain less than $n$ distinct cyclic shifts (see~\cite[p. 105, Lemma 3.20]{Etz24}). Thus, the $(2^k-1,2^k-1,1)$-CSC code
contains much less than $\frac{2^{2^k-k-1}}{2^k-1} + 2^{2^{k-1}}$ codewords and hence asymptotically it is an optimal CSC.

We can make a more precise computation when $n=2^k-1$ is a prime. In this case only two codewords, the all-zero codeword
and the all-one codeword, have less than $n$ distinct shifts. Therefore, the number of codewords in the $(2^k-1,2^k-1,1)$-CSC
is $\frac{2^{2^k-k-1}-2}{2^k-1} +2$. To obtain a $(2^k-1,1)$-CS we have to add to each sequence at most $2^k-2$ bits to convert it to acyclic
sequence which contains all its words of length $2^k-1$. The $(2^k-1,1)$-CS that is generated has length shorter than $2^{2^k-k}$,
i.e., within factor of at most 2 from optimality.

\section{Construction from Self-Dual Sequences}
\label{sec:self-dual}

In this section, we continue with the construction of cyclic codes as presented in Section~\ref{sec:cyclic} by using
a specific family of sequences called self-dual sequences. A self-dual sequence is a cyclic sequence $\cS$ whose complement
$\bar{\cS}$ is the same sequence as $\cS$ after another shift. The construction will be described for $n$ that it is
a power of two, but it can be described to other values of $n$.
For a given word of length~$n$ we are interested in self-dual sequences
of length $2n$. Such sequences are of the form $[X ~ \bar{X}]$, where $X$ is a sequence of length $n$ and $[X ~ \bar{X}]$ has no periodicity since
its length is a power of 2.

The construction for self-dual sequences starts with a $(n,2n,1)$-CSC $\cC$, where $n=2^k$,
(with $2^{2^{k+1}-k-2}$ $n$-tuples), in which for each $n$-tuple $U$ of a codeword in $\cC$
there exists exactly one other $n$-tuple $V$ of a codeword in $\cC$ such that $d(U,V)=1$. For this purpose,
the code $\cC$ will be designed in such a way
that if $[ X \bar{X}]$ is a codeword in $\cC$, then there exists exactly one codeword $[Y , \bar{Y} ]$ in $\cC$ such that $d(X,Y)=1$.
This immediately implies the following observation proved also in~\cite{BER25}.

\begin{lemma}
The $n$-tuples in the codeword of $\cC$ form an $(n,1)$-covering code $\cC$ for which each codeword $c \in \cC$ there
exists exactly one other codeword $c' \in \cC$ such that $d(c,c')=1$.
\end{lemma}

Now, we are in a position to present a construction for an $(n,1)$-CS based on self-dual sequences of length $2n$.

\begin{construction}
\label{con:selfdual}
Let $\cC_n$ be an $(n,2n,1)$-CSC code, $n=2^k$, with $2^{2^k -2k-1}$ sequences, all of them are self-dual. Moreover, if $c \in \cC$ then there exists
a shift $[X~ \bar{X}]$ of $c$ for which $\cC$ has exactly one other codeword $[X' ~ \bar{X}']$, where $X'$ differs from $X$ only in the last coordinate.

Let $\cE_n$ be the set of all $2^{n-2}$ even-weight words in $\F_2^n$ that start with a \emph{zero}.
Let $\cC_{2n}$ be the code with self-dual sequences of length $4n$ defined by
\begin{equation*}
\cC_{2n} \triangleq \{ [ U ~ U+X ~ \bar{U} ~ \bar{U} +X] ~:~ U \in \cE_n, ~ [X ~ \bar{X}] \in \cC_n]]  \},
\end{equation*}
where the basis is
$$
\cC_8 = { [0001101111100100] , [0001101011100101]}
$$
\hfill\quad $\blacksquare $
\end{construction}

The computation of the sizes of the codes and the correctness of the details can be found in~\cite{BER25}, where the
following result was proved.
\begin{theorem}
The code $\cC_{2n}$, $n=2^k$, of Construction~\ref{con:selfdual} is an $(2n,4n,1)$-CSC with $2^{2^{k+1}-2k-3}$ codewords which are
self-dual sequences of length $2^{k+1}$. If a codeword is in $\cC_{2n}$ then it has the form $[ X ~ \bar{X}]$ and
$[X' ~ \bar{X}']$ is another codeword in $\cC_{2n}$.
\end{theorem}

\begin{remark}
Construction~\ref{con:selfdual} is very similar to the constructions presented in~\cite{Etz84,Etz96}.
The construction based on self-dual sequences was found by one of the authors and was motivated by the introduction of nearly-perfect
covering codes~\cite{BER25}. It was first introduced for the earlier conference version of this paper~\cite{CETV24}.
\end{remark}

From the two codewords of $\cC_8$, we can form the optimal $(8,1)$-CS of length 32
$$
[0001101111100100 ~ 0001101011100101] ~.
$$
Construction~\ref{con:selfdual} is applied on $\cC_8$ to obtain a $(16,32,1)$-CSC with 128 codewords
of length 32. These 128 codewords are partitioned into 64 pairs of the form
$$
[X ~ \bar{X}] ~~~ \text{and} ~~~ [X' ~ \bar{X}']~.
$$
Each pair is combined into one sequence of length 64 that contains all the 16-tuples of the original two codewords of length 32.
Hence, this code is an optimal $(16,64,1)$-CSC of length 64 with 64 codewords.
Trivially, each sequence can be extended
with its first 15 bits to an acyclic sequence of length ${64+15=79}$. In other words, these 64 sequences can be concatenated to form
a $(16,1)$-CS of length ${64 \cdot (64+15)=5056}$. However, we can merge these sequences with overlaps between
the suffixes and the prefixes of the sequences and obtain a $(16,1)$-CS of length~4462 (see Appendix~\ref{sec:SD16}),
where the known lower bound, for a $(16,1)$-covering code,
is 4096. The same idea can be applied recursively on the 64 pairs of sequences of length~32
to obtain an upper bound on the shortest length of a $(2^k,1)$-CS.

For $n=2^k$, $k>2$, we obtain a $(2^k,2^{k+2},1)$-CSC of length $2^{k+2}$ with $2^{2^k-2k-2}$ codewords.
This code is optimal as the total length of the codewords is the same as the number of codewords in an optimal $(2^k,1)$-covering code.

Without considering any overlap we can merge these $2^{2^k-2k-2}$ sequences, whose length is $2^{k+2} +2^k-1$ (the $n-1$ prefix
is concatenated after the last bit of each sequence) as an acyclic sequence,
to obtain a $(2^k,1)$-CS of length $2^{2^k-2k-2} (2^{k+2} +2^k -1)$. We can also compute some overlaps between the
prefixes and the suffixes of these sequences; however, we omit this computation due to its complexity. The number of codewords in
an optimal $(2^k,1)$-covering code is $K=2^{2^k-k}$, so this sequence has a length of less than $1.25 K$, i.e., within factor of at most 1.25
from optimality.

\section{Interleaving of Covering Sequences}
\label{sec:interleave}

Interleaving is probably the most simple method to construct $(n,R)$-CSs with a large length and a large covering radius
from two or more covering sequences of smaller length and a smaller covering radius.
When two covering sequences participate in the construction, one is an $(n_1,R_1)$-CS of length $k_1$ and the other is an
$(n_2,R_2)$-covering sequence of length $k_2$, where $n_2 \leq n_1 \leq n_2 +1$
and $\gcd (k_1,k_2)=1$. From these two covering sequences we generate an $(n_1+n_2,R_1 + R_2)$-CS of length $2k_1 k_2$.

\begin{construction}
\label{con:interleave}
Let $\cA = [a_0,a_1, \ldots , a_{k_1-1}]$ be an $(n_1,R_1)$-CS and let
$\cB = [b_0,b_1, \ldots , b_{k_2-1}]$ be an $(n_2,R_2)$-CS and assume further that
$n_1=n_2$ or $n_1 = n_2 +1$ and $\gcd(k_1,k_2)=1$.
The interleaving of $\cA$ and $\cB$, is the sequence
$\cS=[s_0,s_1,\ldots,s_{2k_1 k_2 -1}]$ of length $2k_1k_2$ defined by
$s_0=a_0$, $s_1=b_0$, $s_2=a_1$, $s_3=b_1$, and in general $s_{2i}=a_i$, where $i$ is taken modulo $k_1$
and $s_{2i+1}=b_i$, where $i$ is taken modulo $k_2$, for $0 \leq i \leq k_1 k_2 -1$.
\hfill\quad $\blacksquare $
\end{construction}

\begin{theorem}
\label{thm:interleave}
If $n_1=n_2$ or $n_1=n_2+1$, then the sequence $\cS$ defined in Construction~\ref{con:interleave}
is an ${(n_1+n_2,R_1 + R_2)}$-CS of length $2k_1 k_2$.
\end{theorem}
\begin{proof}
First note that since $\gcd(k_1,k_2)=1$, it follows that for each $i$ and $j$, $0 \leq i \leq k_1-1$,
$0 \leq j \leq k_2 -1$ the string $a_i b_j a_{i+1} b_{j+1} \cdots a_{i+n_1} b_{j+n_1} a_{i+n_1+1}$, where indices
of $\cA$ are taken modulo $k_1$ and indices of $\cB$ are taken modulo $k_2$, is a subsequence in $\cS$.
This also implies that the length of the sequence is $2k_1 k_2$.
We distinguish between the two cases of $n_1=n_2$ and $n_1=n_2+1$.

\noindent
{\bf Case 1:} If $n_1=n_2=n$, then consider a word
$(\alpha_1,\beta_1,\alpha_2,\beta_2,\ldots,\alpha_n,\beta_n)$. Since $\cA$ is an $(n,R_1)$-CS
and $\cB$ is an $(n,R_2)$-CS,
it follows that there exists a subsequence $a_i a_{i+1} \cdots a_{i+n-1}$ of $\cA$ whose distance
from $\alpha_1,\alpha_2,\ldots,\alpha_n$ is at most $R_1$ and there exists
a subsequence $b_j b_{j+1} \cdots b_{j+n-1}$ of $\cB$ whose distance
from $\beta_1,\beta_2,\ldots,\beta_n$ is at most $R_2$.
Therefore, the subsequence of $\cS$, $a_i b_j a_{i+1} b_{j+1} \cdots a_{i+n-1} b_{j+n-1}$
is within distance at most $R_1+R_2$ from $(\alpha_1,\beta_1,\alpha_2,\beta_2,\ldots,\alpha_n,\beta_n)$.

\noindent
{\bf Case 2:}
If $n_1=n_2+1$, then consider a word
$(\alpha_1,\beta_1,\alpha_2,\beta_2,\ldots,\alpha_{n_2-1},\beta_{n_2-1},\alpha_{n_2},\beta_{n_2},\alpha_{n_1})$.
Since we have that $\cA$ is an $(n_1,R_1)$-CS and $\cB$ is an $(n_2,R_2)$-CS,
it follows that there exists a subsequence $a_i a_{i+1} \cdots a_{i+n_1-1}$ of $\cA$ whose distance
from $\alpha_1,\alpha_2,\ldots,\alpha_{n_1}$ is at most $R_1$ and there exists
a subsequence $b_j b_{j+1} \cdots b_{j+n_2-1}$ of $\cB$ whose distance
from $\beta_1,\beta_2,\ldots,\beta_{n_2}$ is at most $R_2$.
Therefore, the subsequence of $\cS$, $a_i b_j a_{i+1} b_{j+1} \cdots a_{i+n_2-1} b_{j+n_2-1} a_{i+n_1-1}$
is within distance at most $R_1+R_2$ from the word $(\alpha_1,\beta_1,\alpha_2,\beta_2,\ldots,\alpha_{n_2},\beta_{n_2},\alpha_{n_1})$.
\end{proof}

\begin{example}
\label{ex:inter_r=1}

The following $(n,R)$-CSs were obtained by using
Construction~\ref{con:interleave}
on two shorter covering sequences which are either trivial or presented in Appendix~\ref{sec:verysmall} or in Appendix~\ref{sec:small}.


Consider the $(9,0)$-CS of length 512 and
a $(9,1)$-CS length 93. Interleaving these two sequences using Construction~\ref{con:interleave} yields
a $(18,1)$-CS of length $2 \cdot 512 \cdot 93 = 95232$.

Consider the $(10,0)$-CS of length 1024 and
a $(10,1)$-CS length 175. Interleaving these two sequences using Construction~\ref{con:interleave} yields
a $(20,1)$-CS of length $2 \cdot 1024 \cdot 175 =358400$.

Consider an $(8,1)$-CS of length 32 and
a $(9,1)$-CS length 93. Interleaving these two sequences using Construction~\ref{con:interleave} yields
a $(17,2)$-CS of length $2 \cdot 32 \cdot 93 =5952$.

Consider an $(8,1)$-CS of length 37 and
an $(8,2)$-CS length 14. Interleaving these two sequences using Construction~\ref{con:interleave} yields
a $(16,3)$-CS of length $2 \cdot 37 \cdot 14 =1036$.

Consider an $(8,1)$-CS of length 37 and
a $(9,2)$-CS length 20. Interleaving these two sequences using Construction~\ref{con:interleave} yields
a $(17,3)$-CS of length $2 \cdot 37 \cdot 20 =1480$.

Consider a $(9,1)$-CS of length 93 and
a $(9,2)$-CS length 20. Interleaving these two sequences using Construction~\ref{con:interleave} yields
a $(18,3)$-CS of length $2 \cdot 93 \cdot 20 =3720$.

Consider a $(9,1)$-CS of length 93 and
a $(10,2)$-CS length 38. Interleaving these two sequences using Construction~\ref{con:interleave} yields
a $(19,3)$-CS of length $2 \cdot 93 \cdot 38 =7068$.

Consider a $(10,1)$-CS of length 175 and
a $(10,2)$-CS length 38. Interleaving these two sequences using Construction~\ref{con:interleave} yields
a $(20,3)$-CS of length $2 \cdot 175 \cdot 38 =13300$.

\hfill\quad $\blacksquare $
\end{example}

Construction~\ref{con:interleave} can be generalized to three or more sequences. When $t$ sequences are interleaved
the requirement is that each sequence is an $(n_i,R_i)$-CS of length $k_i$, $1 \leq i \leq t$, where
$\gcd(k_i,k_j)=1$, $1 \leq i < j \leq t$, and $n \leq n_i \leq n+1$ for some positive integer $n$.
From these $t$ sequences we generate an $(\sum_{i=1}^t n_i , \sum_{i=1}^t R_i )$-CS sequence
of length $\ell=t\prod_{i=1}^t k_i$, $[s_0 s_1 ~ \cdots ~ s_{\ell-1}]$, where $s_r$, $r \equiv j ~ (\mmod ~ t)$,
$1 \leq j \leq t$, is a bit from the $j$-th sequence. Unfortunately, the factor $t$ in the length of the sequence
makes this construction quite weak for $t>2$ and also for $t=2$ (see Construction~\ref{con:interleave}) it is not too effective, but some of our best
$(n,R)$-CSs for $n \leq 20$ and $1 \leq R \leq 3$ are obtained by this construction.
Fortunately, we can get rid of the factor 2 when $t=2$ and $n_1=n_2$ at the expense of some extra small redundancy as follows.

\begin{construction}
\label{con:inter2}
Let $\cA = [a_0,a_1, \ldots , a_{k-1}]$ be an $(n,R)$-CS in which $a_i =0$ for $0 \leq i \leq n-2$,
i.e., $\cA$~has a subsequence with $n-1$ consecutive \emph{zeros} (the same can be applied for the \emph{ones}).
If $k$ is even then form the following sequence $\cS$ presented in $k/2$ parts.
The first part is
$$
a_0 a_0 a_1 a_1 ~ \cdots ~ a_{k-1} a_{k-1} a_0 0 ~.
$$
The second part is
$$
a_1 a_0 a_2 a_1 ~ \cdots ~ a_{k-1} a_{k-2} a_0 a_{k-1} a_1 0 ~.
$$
The $i$-th part, $1 \leq i \leq k/2$, is
$$
a_{i-1} a_0 a_i a_1 a_{i+1} a_2 ~ \cdots ~ a_{i-2} a_{k-1} a_{i-1} 0 ~.
$$
These $k/2$ parts, each one of length $2k+2$, are concatenated together, in their order, to the final sequence~$\cS$.

If $k$ is odd, then form the same parts, where the last one is for $i= \frac{k+1}{2}$.
\hfill\quad $\blacksquare $
\end{construction}

\begin{theorem}
\label{thm:inter2}
The sequence $\cS$ defined in Construction~\ref{con:inter2} is an $(2n,2R)$-CS of length $k(k+1)$ if $k$ is even
and $(k+1)^2$ if $k$ is odd.
\end{theorem}
\begin{proof}
The proof is essentially as the one of Theorem~\ref{thm:interleave}. The main difference is that the two sequences $\cA$ and $\cB$
in Construction~\ref{con:interleave} is replaced with only one sequence in Theorem~\ref{thm:interleave}. The sequence $\cB$ is replaced
by $\cA$ to which a \emph{zero} is concatenated at the end. All the associated shifts
between the different positions of the sequence are guaranteed by the different parts. The 0 at the end of each part makes sure that
we will move to another shift of $\cA$. It does not destroy any $n$-tuples since it can be considered as adding another
\emph{zero} to the run of $n-1$ \emph{zeros} in one of the interleaved sequences which does not damage any of the covered words of length $n$.
\end{proof}

\begin{example}
\label{ex:inter_r=2}

The following $(n,2)$-CSs were obtained by using Construction~\ref{con:inter2}.

Consider a $(8,1)$-CS of length 40 with a subsequence of 7 consecutive \emph{zeros} apply Construction~\ref{con:inter2}
to obtain a $(16,2)$-CS of length $40 \cdot 41 =1640$.

Consider a $(9,1)$-CS of length 102 with a subsequence of 8 consecutive \emph{ones} apply Construction~\ref{con:inter2}
to obtain an $(18,2)$-CS of length $102 \cdot 103 =10506$.

Consider a $(10,1)$-CS of length 177 with a subsequence of 9 consecutive \emph{zeros} apply Construction~\ref{con:inter2}
to obtain a $(20,2)$-CS of length $178 \cdot 178 =31684$.

\hfill\quad $\blacksquare $
\end{example}
\section{A Construction from Primitive Polynomials}
\label{sec:primitive}

A construction based on a specific type of primitive polynomials usually does not generate a very short sequence.
However, one of our specific bounds (see Table~\ref{table:one_dim_2}) comes from such a sequence.
Shift-register sequences and especially those which are produced by a linear function or a modification of such
function are very useful for various applications and for constructing of other related combinatorial structures~\cite{Etz24,Gol67}.
The same approach is taken in this section. Of special interest are binary sequences generated by a linear function
and whose length is $2^n-1$. These sequences are called M-sequences (for maximal length) and their function is derived
from a coefficient of a primitive polynomial.

Let $c(x) = \sum_{i=0}^n c_i x^i$, where $c_n=c_0=1$ and $c_i \in \{0,1\}$ for $1 \leq i \leq n-1$,
be an irreducible polynomial. Define the following infinite sequence $a_0, a_1,a_2,\ldots$, where
\begin{equation}
\label{eq:rec_Mseq}
a_k = \sum_{i=1}^n c_i a_{k-i}
\end{equation}
with the initial nonzero $n$-tuple $( a_{-n},a_{-n+1},\ldots,a_{-1} )$.
If $c(x)$ is a primitive polynomial, i.e., a polynomial whose roots have order $2^n-1$, then the sequence
$\cA=[ a_0,a_1,a_2 ,\ldots ]$ is called an M-sequence, its period is
period $2^n-1$, and hence only its first $2^n-1$ terms are considered.
In this cyclic sequence each nonzero $n$-tuple appears exactly once as a window of length $n$.
Each root of the primitive polynomial generates the field GF$(2^n)$, i.e., it is a primitive element of the field.
The $2^n-1$ consecutive cyclic shifts of $\cA$ can be viewed as another isomorphic representation of the field,
These details are well-explained in these two books~\cite{Etz24,Gol67}.

Consider now another recursion for a sequence defined by
\begin{equation}
\label{eq:rec_comp_Mseq}
b_k = \sum_{i=1}^n c_i b_{k-i} +1
\end{equation}
with the initial nonzero $n$-tuple $( b_{-n},b_{-n+1},\ldots,b_{-1} )$.

\begin{lemma}
The sum $\sum_{i=1}^n c_i$ is an even integer.
\end{lemma}
\begin{proof}
This follows from the fact that after $n$ consecutive \emph{ones} we should have a \emph{zero}
in such a sequence. Another argument is that if $\sum_{i=1}^n c_i$ is an odd integer, then
$\sum_{i=0}^n c_i=0$, 1 is a root of $c(x)$ and hence $c(x)$ is not a primitive polynomial.
\end{proof}

\begin{lemma}
The sequence $\cB=(  b_0,b_1,b_2 ,\ldots )$ is the binary complement of the sequence $\cA$.
\end{lemma}
\begin{proof}
Let $\blda \triangleq (a_j,a_{j+1},\ldots,a_{j+n-1})$ be an $n$-tuple in the sequence $\cA$. By Equation~(\ref{eq:rec_Mseq})
we have that
$$
a_{j+n}= \sum_{i=1}^n c_i a_{j+n-i}.
$$
Let $\bldb \triangleq (b_j,b_{j+1},\ldots,b_{j+n-1})$ be an $n$-tuple in the sequence $\cB$, where $b_i=a_{i}+1$,
for $j \leq i \leq j+n-1$, i.e., $\bldb = \bar{\blda}$. Therefore, we have that
\begin{align*}
b_{j+n} & = \sum_{i=1}^n c_i b_{j+n-i}+1 = \sum_{i=1}^n c_i (a_{j+n-i}+1) +1 = \sum_{i=1}^n c_i a_{j+n-i} + \sum_{i=1}^n c_i +1 \\
&= \sum_{i=1}^n c_i a_{j+n-i}+1  =a_{j+n} +1 &  .
\end{align*}
This implies that the sequence $\cB$ is the binary complement of the sequence~$\cA$.
\end{proof}

\begin{corollary}
The recursion of Equation~\textup{(\ref{eq:rec_Mseq})} generates the sequence~$\cA$ and the all-zero sequence.
The recursion of Equation~\textup{(\ref{eq:rec_comp_Mseq})} generates the sequence $\cB=\bar{\cA}$ and the all-one sequence.
\end{corollary}

\begin{lemma}
\label{lem:prim_cov}
Let $c(x) = \sum_{i=0}^n c_i x^i$ be a primitive polynomial for which $c_i=0$ for
$1 \leq i \leq 2R+1$ and consider the sequences $\cA$, $\cB$, the all-zeros sequence
and the all-ones sequence. The code that contains these four sequences is an $(n+2R+1,R)$-CSC.
\end{lemma}
\begin{proof}
Let $X$ be any given $n$-tuple.
Since the first $2R+1$ $c_i$s (except for $c_0$) are \emph{zeros}, it follows that the last $2R+1$
elements of $X$ are not influencing the result of the next bit for both recursions and hence
the addition of the 1 in the sequence~$\cB$ implies that the next $2R+1$ bits after $X$ in $\cA$
and~$\cB$ will be complementing. This implies that $X z_1 z_2 ~ \cdots ~ z_{2R+1}$ is in the ball of radius $R$ either
by an $(n+2R+1)$-tuple of~$\cA$ that starts with $X$ or  by an $(n+2R+1)$-tuple of $\cB$ that starts with~$X$.
\end{proof}

The number of sequences of the code defined in Lemma~\ref{lem:prim_cov} is four and they can be efficiently
concatenated to one ${(n+2R+1,R)}$-CS.

\begin{theorem}
\label{thm:priCS}
The four sequences of the $(n+2R+1,R)$-CSC can be combined to a $(n+2R+1,R)$-CS of length $2^{n+1} + 2n+8R+2$.
\end{theorem}
\begin{proof}
The span $n$ M-sequence $\cA$ has length $2^n -1$ and the same length has the sequence $\cB$.
The sequence~$\cA$ has a subsequence of $n-1$ consecutive \emph{zeros} and hence it can be combined with the all-zero sequence
to an acyclic sequence of length $2^n +n+4R+1$. The same can be done for the sequence $\cB$ and the all-one sequence
and the two sequences can be combined with no overlaps to an ${(n+2R+1,R)}$-CS of length $2^{n+1}+2n+8R+2$.
\end{proof}

\begin{remark}
The sequences generated in Theorem~\ref{thm:priCS} can be combined with some overlap to reduce its length, but this will make
only a very small improvement in the length of the covering sequence. It should be noted the required primitive polynomials
exist for relative small $n$ and larger one.
\end{remark}

Finally, in Table~\ref{table:one_dim_2} the current best lower and upper bounds on $\cL(n,R)$ are presented, where the lower bounds are either
by computer search for very small $n$ or the known lower bound on the smallest size of an $(n,R)$-covering code.

\begin{table}[ht!]
	\centering
	\small
\begin{tabular}{|c|c|c|c|c|c|}
\hline
$n$ & $R=1$ & $R=2$& $R =3$   \\
\hline
$9$& 62-93$^{\text{a}}$& 20$^{\text{b}}$ & 12$^{\text{b}}$ \\
\hline
$10$& 107-175$^{\text{a}}$&38$^{\text{b}}$ & 16$^{\text{b}}$ \\
\hline
$11$& 180-283$^{\text{a}}$&38-111$^{\text{a}}$ & 20$^{\text{b}}$ \\
\hline
$12$& 342-597$^{\text{a}}$&62-161$^{\text{a}}$ & 34-40$^{\text{b}}$ \\
\hline
$13$& 598-1172$^{\text{a}}$&97-292$^{\text{a}}$ & 34-93$^{\text{a}}$ \\
\hline
$14$&1172-2271$^{\text{a}}$ &159-525$^{\text{a}}$ & 44-239$^{\text{c}}$ \\
\hline
$15$& 2048-3516$^{\text{e}}$ &310-907$^{\text{a}}$ & 70-406$^{\text{a}}$ \\
\hline
$16$&  4096-4462$^{\text{f}}$&512-1640$^{\text{c}}$ & 115-1036$^{\text{d}}$ \\
\hline
$17$&  7419-17719$^{\text{a}}$& 859-5952$^{\text{d}}$ & 187-1480$^{\text{d}}$ \\
\hline
$18$& 14564-95232$^{\text{d}}$ & 1702-10506$^{\text{c}}$ & 316-3720$^{\text{d}}$ \\
\hline
$19$& 26309-176170$^{\text{g}}$ & 2898-31684$^{\text{h}}$ & 513-7068$^{\text{d}}$\\
\hline
$20$&52618-358400$^{\text{d}}$ & 5330-31684$^{\text{c}}$ & 892-13300$^{\text{d}}$ \\
\hline
\end{tabular}
\caption{Bounds for the length of the shortest $(n,R)$-CS for $9 \leq n \leq 20$ and $R=1,2,3$.}
a - computer search, b - Chung and Cooper~\cite{ChCo04}, c - Construction~\ref{con:inter2}, d - Construction~\ref{con:interleave},
e - from the Hamming code, ~~~~~~~~~   f - from self-dual sequences, g - from primitive polynomial, h - Theorem~\ref{thm:trivialB}.
\label{table:one_dim_2}
\end{table}

\section{Folding of a Covering Sequence into a 2D-Sequence}
\label{sec:folding}

After we have looked at covering sequences and covering sequence codes we continue with a generalization of
the one-dimensional framework into two-dimensional arrays. The ultimate goal is to construct an
$M \times N$ array $\cA$ in which for each $m \times n$ matrix $\cB$ there exists an $m \times n$ submatrix $\cX$
of $\cA$ such that $d(\cB,\cX) \leq R$. Two techniques will be presented in this section. The first one is
a folding technique and the second one involves different related shifts of a one-dimensional covering sequence.

For the first technique, folding, we will use the following simple lemma that was already observed in~\cite{ChCo04}.

\begin{lemma}
\label{lem:extendS}
If there exists an $(n,R)$-CS of length $k$, then there exists an $(n,R)$-CS sequence of length $k+n-1+\epsilon$ for any
$\epsilon \geq 0$.
\end{lemma}
\begin{proof}
If the $(n,R)$-CS $\cS$ starts at a certain position, then we can append to $\cS$ the first $n-1$ bits and any $\epsilon$ bits after that to
keep it an $(n,R)$-CS.
\end{proof}

\begin{construction}
\label{con:folding}
Start with an $(mn,R)$-CS $\cS=s_0,s_1,\ldots,s_{k-1}$ of length $k$, where without loss of generality $k$ is divisible by~$n$
(see Lemma~\ref{lem:extendS}, where $0 \leq \epsilon < n-1$).
The sequence $\cS$ is folded row by row into an $M \times n$ array $\cA$, where $M=\frac{k}{n}$. The array is extended into
an $M \times (2n-1)$ array $\cA'$ by adding to each row of the $M \times n$ array $\cA$ the next $n-1$ symbols from $\cS$ associated
with the given row.
In other words, the $j$-th row of the array $\cA'$, where $0 \leq j \leq \frac{k}{n} -1$, is defined by
\begin{equation}
\label{eq:row_in_array}
s_{jn+1},s_{jn+2},\ldots,s_{jn+n},s_{jn+n+1},\ldots,s_{jn+2n-1},
\end{equation}
where the indices are taken modulo $k$.
\hfill\quad $\blacksquare $
\end{construction}

Before proving that $\cA'$ generated in Construction~\ref{con:folding} is an $(m \times n,R)$-C2DS we will prove
one property of the array $\cA'$.
\begin{lemma}
\label{lem:subArrayA}
The array $\cA'$ generated in Construction~\ref{con:folding} contains the $m \times n$ subarray
$$
\begin{array}{cccc}
s_{jn+i} & s_{jn+i+1} &  \cdots  & s_{jn+i+n-1} \\
s_{(j+1)n+i} & s_{(j+1)n+i+1} &  \cdots & s_{(j+1)n+i+n-1} \\
\vdots & \vdots & \ddots & \vdots \\
s_{(j+m-1)n+i} & s_{(j+m-1)n+i+1} &  \cdots &  s_{(j+m-1)n+i+n-1},
\end{array}
$$
where $0 \leq i \leq n-1$, $0 \leq j \leq \frac{k}{n} -1$, and indices are taken modulo $k$.
\end{lemma}
\begin{proof}
This is an immediate consequence of Equation~(\ref{eq:row_in_array}) which defines the $j$-th row of the array $\cA'$.
\end{proof}
\begin{corollary}
\label{cor:subArrayA}
Each $m \times n$ matrix obtained by folding $mn$ consecutive bits of $\cS$, row by row, is contained as an $m \times n$ subarray of $\cA'$.
\end{corollary}

\begin{theorem}
The array $\cA'$ generated in Construction~\ref{con:folding} is an $(m \times n,R)$-C2DS.
\end{theorem}
\begin{proof}
Given an $m \times n$ matrix $\cB$ we have to show that there exists an $m \times n$ subarray $\cX$ of $\cA'$ such that $d(\cB,\cX) \leq \cR$.
Let $X_1, X_2, \ldots , X_m$ be the $m$ consecutive rows of $\cB$ and let $\cT = X_1 X_2 ~ \cdots ~ X_m$ the sequence of length $mn$ obtained
by concatenating them. The sequence $\cT$ has length $nm$ and hence there exists a subsequence $\cY$ in $\cS$ such that $d(\cT,\cY)\leq R$.
By corollary~\ref{cor:subArrayA} the $m \times n$ matrix $\cY'$ obtained by folding $\cY$ into an $m \times n$ matrix is contained
in $\cA'$. Since $d(\cT,\cY)\leq R$ it follows that $d(\cB,\cY')\leq R$ which completes the proof.
\end{proof}

Construction~\ref{con:folding} implies the following consequence.
\begin{theorem}
If there exists an $(mn,R)$-CS of length $k$, where $n$~divides~$k$, then there exists a
$(m \times n,R)$-C2DS of size $M \times N$, where $M= \frac{k}{n}$ and $N=2n-1$.

If there exists an $(mn,R)$-CS of length $k$, where $n$ does not divide $k$, then there exists
an $(m \times n,R)$-C2DS of size $M \times N$,
where $N=2n-1$ and $\frac{k}{n} \leq M \leq \left\lceil \frac{k}{n} \right\rceil +m$.
\end{theorem}

The main disadvantage of Construction~\ref{con:folding} is that the width of the array $N=2n-1$ is
not large enough compared to the width of the window $n$.
Next, we will be interested in starting with an $(mn,R)$-CS sequence
${\cS=[s_0,s_1,s_2,\ldots,s_{k-1} ]}$ and
construct an $M \times N$ $(m \times n,R)$-C2DS for which $M$ will be considerably larger than $m$ and $N$ considerably larger
than $n$ and the ratio $\frac{MN}{k}$ should be as small as possible (i.e., with as little redundancy as possible).
There are a few ways to construct such an array using the same principles as in Construction~\ref{con:folding}.
The sequence is partitioned into $r$ subsequences which form an $(mn,R)$-CSC. The $r$ sequences should be of the same
length $\kappa$ which is divisible by $n$. From each sequence we create an $M \times (2n-1)$ array and we concatenate these
$r$ arrays into one $M \times (2nr-r)$ array which is a $(m \times n,R)$-C2DS. The specific details are omitted as they are
the same as in Construction~\ref{con:folding}.

The folding construction is very simple and asymptotically, it yields optimal $(m \times n,R)$-C2DS if the $(mn,R)$-CS is asymptotically optimal.
As we mentioned, Proposition~\ref{prop:prob_bound} can be obtained by using the folding construction, along with the covering sequence
whose existence was proved by Vu~\cite{vu05}. The length of this $(mn,R)$-CS is close up to a factor of $\log mn$ from the sphere-covering bound.

It is natural to think that an $(mn,R)$-CS of length $k_1$, will have a smaller length than the area of an $(m \times n,R)$-C2DS with
dimension $M \times N$, i.e., $k_1 \leq M \cdot N$. This is indeed the case when folding is applied, $M \cdot N$ is about twice as
large as $k_1$. However, we provide one more construction for
covering 2D-sequences from covering sequences, whose outcome is quite surprising as it produces smaller arrays (in area)
than the length of the ones obtained by folding the related sequence. Moreover, sometimes it yields
an array whose area is smaller than the associated best known covering sequence.
This second technique is done by considering different shifts of an one-dimensional $(n,R)$-CS. It will be presented in two
constructions, where the first one is very simple.

\begin{construction}
\label{con:shifts}
Let $\cS$ be an $(n,R)$-CS of length $k$.

If $k$ is even, then form an $(k+1) \times k$ array
whose $i$-th row, $0 \leq i \leq k-1$ is ${\bf E}^j S$, where $j = \sum_{\ell=0}^i \ell$,
${\bf E}^jS$ is a cyclic shift of $\cS$ by $j$ positions to the left, i.e.,
$$
{\bf E}^j [s_0,s_1,\ldots, s_{k-1}] = [s_j,s_{j+1},\ldots,s_{k-1},s_0,\ldots,s_{j-1}],
$$
and the $k$-th row is the same as the $(k-1)$th row.

If $k$ is odd, then form an $k \times k$ array whose $i$-th row, where
$0 \leq i \leq k-1$ is ${\bf E}^j S$, $j = \sum_{\ell=0}^i \ell$.
\hfill\quad $\blacksquare $
\end{construction}

\begin{theorem}
\label{thm:shifts}
If $n$ is even, then the $(k+1) \times k$ array obtained in Construction~\ref{con:shifts} is
an $(2 \times n,2R)$-C2DS.
If $n$ is odd, then the $k \times k$ array obtained in Construction~\ref{con:shifts} is
an $(2 \times n,2R)$-C2DS.
\end{theorem}
\begin{proof}
Let $\cA$ a $k \times k$ array obtained in Construction~\ref{con:shifts} and let $\cB$ a $2 \times n$
matrix that consists of two sequences (rows) of length $n$, the first one $X$ and the second one $Y$.
Each row of $\cA$ is the sequence~$\cS$ at some shift, where the sequence $\cS$ of the $i$-th row, $1 \leq i \leq k$,
is shifted by $i$ positions to the left compared to the sequence in the $(i-1)$-th row. The 0-th row is taken without shifting.

Assume first that $k$ is even. Since $\cS$ is an $(n,R)$-CS,
it follows that there exists a subsequence $U$ of length $n$ in $\cS$ such the $d(X,U) \leq R$ and
a subsequence $V$ of length $n$ in $\cS$ such that $d(Y,V) \leq R$.
Since $\cS$ is an $(n,R)$-CS and all possible shifts are taken between consecutive rows,
it follows that in two consecutive rows of $\cA$ we have the subsequence $U$ and $V$ one on top of the other,
i.e., a $2 \times n$ matrix $\cW$ such that $d(\cW,\cB) \leq 2R$.
Thus, the $(k+1) \times k$ array obtained in Construction~\ref{con:shifts} is
an $(2 \times n,2R)$-C2DS.

When $k$ is odd the proof is similar, but we have to notice that since $\sum_{i=1}^k i \equiv 0 ~(\mmod ~ k)$,
it follows that the 0-th row is without any shift compared to the last row and hence one
less row is required.
\end{proof}

\begin{example}
\label{ex:multi_shift6}
Consider the $(6,1)$-CS of length 12 presented in Appendix~\ref{sec:verysmall}. By applying the
defined shifts of Construction~\ref{con:shifts} on this sequence which is the first row in the array we obtain a $13 \times 12$ $(2 \times 6,2)$-C2DS
whose area is 156. The best associated $(12,2)$-CS that we find is based on computer search
has length 161 (see Appendix~\ref{sec:small}).
The generated $13 \times 12$ $(2 \times 6,2)$-C2DS whose area is 156 which smaller than the related $(12,2)$-CS of length 161 that
was found by computer search. It is the following array:

$$
\begin{array}{cccccccccccc}
0&0&0&1&0&0&1&1&1&0&1&1 \\
0&0&1&0&0&1&1&1&0&1&1&0 \\
1&0&0&1&1&1&0&1&1&0&0&0 \\
1&1&1&0&1&1&0&0&0&1&0&0 \\
1&1&0&0&0&1&0&0&1&1&1&0 \\
1&0&0&1&1&1&0&1&1&0&0&0 \\
0&1&1&0&0&0&1&0&0&1&1&1 \\
0&0&1&1&1&0&1&1&0&0&0&1 \\
0&0&0&1&0&0&1&1&1&0&1&1 \\
0&1&1&0&0&0&1&0&0&1&1&1 \\
1&1&0&1&1&0&0&0&1&0&0&1 \\
1&1&1&0&1&1&0&0&0&1&0&0 \\
1&1&1&0&1&1&0&0&0&1&0&0
\end{array}
$$
\hfill\quad $\blacksquare $
\end{example}

\begin{example}
\label{ex:multi_shift7}
Consider the $(7,1)$-CS of length 22 presented in Appendix~\ref{sec:verysmall}. By applying the
defined shifts of construction~\ref{con:shifts} on this sequence which is the first row in the array, we obtain a $23 \times 22$ $(2 \times 7,2)$-C2DS
whose area 506. The best related $(14,2)$-CS that we found is based on computer search
has length 525 (see Appendix~\ref{sec:small}).

\hfill\quad $\blacksquare $
\end{example}


Finally, Construction~\ref{con:shifts} can be generalized as follows.

\begin{construction}
\label{con:largeW}
Let $\cS$ be an $(n,R)$-CS of length $k$ and let $\cT=[t_1,t_2, \ldots ,t_{k^{m-1}}]$ be
a span $m-1$ de Bruijn sequence over $\Sigma_k$.
We form the following $k^{m-1} \times k$ array $\cA$, where the sequence $\cS$ in the
$i$-th row of $\cA$, $1 \leq i \leq k^{m-1}$ is shifted $t_i$ positions related to the sequence $\cS$ in the $(i-1)$-th row,
where the 0-th row is considered to be with no shift.
\hfill\quad $\blacksquare $
\end{construction}

\begin{theorem}
\label{thm:largeW}
The $k^{m-1} \times k$ array $\cA$ of Construction~\ref{con:largeW} is an $(m \times n,mR)$-C2DS.
\end{theorem}
\begin{proof}
The proof is very similar to the one of Theorem~\ref{thm:shifts}. First note that the sum of all the shifts is zero and hence
the virtual 0-th row can be considered to be with no shifts.
For any $m \times n$ matrix $\cB$ we have to find an $m \times n$ window $\cX$ in $\cA$ such that $d(\cB,\cW) \leq mR$.
Let $B_1, B_2, \ldots, B_m$ be the $m$ consecutive rows of $\cB$. Each $B_i$ has length $n$ and hence there exists
a subsequence $X_i$ of length $n$ in $\cS$ such that $d(B_i,X_i)\leq R$.
Let $(i_1,i_2,\ldots,i_{m-1})$ be the $m-1$ consecutive shifts of $\cS$ such that
these $m$ copies of~$\cS$ placed on each other contain $\cB$ as a window. Since $(i_1,i_2,\ldots,i_{m-1})$ is
an $(m-1)$-tuple over $\Z_k$ and $\cT$~is a span $m-1$ de Bruijn sequence over $\Z_k$ which implies that
$(i_1,i_2,\ldots,i_{m-1})$ is contained in $\cT$, it follows that the associated shifts of $\cS$ are contained in $m$ consecutive
rows of $\cA$. These shifts contain an $m \times n$ window $\cX$ whose consecutive rows are $X_1,X_2,\ldots, X_m$.
Since $d(X_i,B_i) \leq R$, it follows that $d(\cB,\cX)\leq mR$.

Thus, the $k^{m-1} \times k$ array $\cA$ of Construction~\ref{con:largeW} is an $(m \times n,mR)$-C2DS.
\end{proof}

\section{Conclusion and Future Research}
\label{sec:conclude}

Covering sequences and covering sequence codes which generalize the well-known covering codes, were considered.
Some new construction methods
for such covering sequences as well as covering sequence codes were presented. In particular, nearly optimal
and asymptotically optimal sequences and codes were obtained using the Hamming codes and self-dual sequences.
Finally, generalization for covering 2D-sequences was also discussed and related constructions were given.

This area is far from being fully explored and we conclude with the following problems for future research.

\begin{enumerate}
\item The current work and also all the previous papers that considered this topic, have concentrating on the upper bounds
of $(n,R)$-CSs. The lower bounds mentioned in this paper are either the ones used for $(n,R)$-covering codes
or obtained by computer search.
We would like to see improvements in these lower bounds and not just by one or two (something which is not difficult to do).

\item Many of the upper bounds used in Table~\ref{table:one_dim_2} are quite weak as Construction~\ref{con:interleave} is used and we
do not have for these parameters a stronger construction like Construction~\ref{con:inter2}. We would like to see some
new constructions that will make the table more balanced. Similarly, we want to see upper bounds on the sizes
of covering sequence codes.

\item The $(2^k-1,1)$-CS and the $(2^k,1)$-CS introduced in Sections~\ref{sec:cyclic} and~\ref{sec:self-dual}
are the best covering sequences obtained for small radii.
We would like to have a more precise computation on the length of these sequences.
We would also like to see similar sequences for radius 2 and radius 3. The Preparata code used
to obtain covering codes for radii 2 and 3~\cite{EtGr93,EtMo05} might be the ones to use for this purpose.

\item We would like to see more constructions as well as lower bounds on the size of covering 2D-sequences.

\item We would like to see a more comprehensive study on covering sequence codes, mainly $(n,m,R)$-CSCs, as the current work is mainly on
covering sequences.

\item In the same way that covering codes are defined on other metrics,
different from the Hamming that was discussed in the paper, covering sequences can be defined for these metrics.
For example, it would be interesting to have such short sequences for various poset spaces.
\end{enumerate}

Recently, Rosin~\cite{Ros25} develop a method to search for combinatorial structures. His search found many new such structures
including some better covering sequences. The sequences themselves are not yet well documented and we hope he will elaborate on this and
improves some of our bounds soon.

\appendices


\section{Very small $(n,R)$-CSs used for other bounds}
\label{sec:verysmall}

%
%
%
%

An $(8,1)$-CS sequence of length 32 - $[00011011111001000001101011100101]$.

An $(8,1)$-CS sequence of length 35 - $[00010110110111000010001111011101001]$.

An $(8,1)$-CS sequence of length 37 - $[0001101111100100000110101110011100101]$.

An $(8,1)$-CS sequence of length 40 (with 7 consecutive \emph{zeros}) -\\ ~ \hspace{44cm}   $~~~~~~~~~~~~~~~~~[0001101111100100000001000001101011100101]$.


An $(8,2)$-CS sequence of length 14 - $[00111011010010]$.

An $(9,2)$-CS sequence of length 20 - $[00010010001110110111]$.


\section{Small $(n,R)$-CSs}
\label{sec:small}

\begin{scriptsize}
The following thirteen codewords form a $(10,11,1)$-CSC:
\begin{align*}
& [00001010000], [00101001011], [10100101111], [10110111001],[11011101111],\\
& [11011101111], [01110111100], [11110010011], [11110010000], [11001000101],\\
& [00100011000], [01000110101], [10001101001], [00011010000] \, .
\end{align*}
Their extension by nine bits and their associated overlaps are as follows:
\begin{multicols}{3}
\begin{tabular}{rl}
  00001010000000010100 & 7 \\
  00101001011001010010 & 7 \\
  10100101111101001011 & 4 \\
  10110111001101101110 & 7 \\
  11011101111110111011 & 7 \\
\end{tabular}

\columnbreak

\begin{tabular}{rl}
  01110111100011101111 & 4 \\
  11110010011111100100 & 9 \\
  11110010000111100100 & 7 \\
  11001000101110010001 & 7 \\
  00100011000001000110 & 8 \\
\end{tabular}

\columnbreak

\begin{tabular}{rl}
  01000110101010001101 & 8 \\
  10001101001100011010 & 8 \\
  00011010000000110100 & 2 \\
\end{tabular}
\end{multicols}
The total number of bits in the thirteen sequences is 260 and the total number of overlaps is 85.
This yields a $(10,1)$-CS of length $260-85=175$:
\begin{align*}
[&00001010000000010100101100101001011111010010\\
&11011100110110111011111101110111100011101111\\
&00100111111001000011110010001011100100011000\\
&0010001101010100011010011000110100000001101] \, .
\end{align*}
The following thirteen codewords form a $(10,11,1)$-CSC:
\begin{align*}
&[11010111111], [01011110000], [10111100101], [11100110011], [11001100001] \\
& [00110001101], [10001101100], [11011010101], [01101010010], [10101000011],\\
& [10000010010], [00000100000], [00001000101] \, .
\end{align*}
Their extension by nine bits and their associated overlaps are as follows:
\begin{multicols}{3}
\begin{tabular}{rl}
 11010111111110101111 & 7 \\
 01011110000010111100 & 8 \\
 10111100101101111001 & 6 \\
 11100110011111001100 & 8 \\
 11001100001110011000 & 7 \\
\end{tabular}

\columnbreak

\begin{tabular}{rl}
 00110001101001100011 & 6 \\
 10001101100100011011 & 5 \\
 11011010101110110101 & 7 \\
 01101010010011010100 & 7 \\
 10101000011101010000 & 5 \\
\end{tabular}

\columnbreak

\begin{tabular}{rl}
  10000010010100000100 & 8 \\
 00000100000000001000 & 8 \\
 00001000101000010001 & 1 \\
\end{tabular}
\end{multicols}
The total number of bits in the thirteen sequences is 260 and the total number of overlaps is 83. This yields a $(10,1)$-CS of length $260-83=177$
with a subsequence having 10 consecutive \emph{zeros}:
\begin{align*}
[&11010111111110101111000001011110010110111100\\
&11001111100110000111001100011010011000110110\\
&01000110110101011101101010010011010100001110\\
&101000001001010000010000000000100010100001000] \,.
\end{align*}
The following twenty codewords form a $(11,11,1)$-CSC:
\begin{align*}
&[00111011011], [01110110100], [10110101100], [10101100010], [10001001011], \\
&[10001001010], [01010100111], [01010011010], [01001101110], [10111111110], \\
&[10111111111], [11111111001], [11111000010], [11100001100], [00011011001], \\
&[01100000000], [00000010000], [00001000001], [00000111101], [00011110011] \, .
\end{align*}
Their extension by ten bits and their associated overlaps are as follows:
\begin{multicols}{3}

\noindent
\begin{tabular}{rl}
  001110110110011101101 & 9 \\
  011101101000111011010 & 7 \\
  101101011001011010110 & 7 \\
  101011000101010110001 & 5 \\
  100010010111000100101 & 10 \\
  100010010101000100101 & 4 \\
  010101001110101010011 & 8 \\
\end{tabular}

\columnbreak

\begin{tabular}{rl}
  010100110100101001101 & 8 \\
  010011011100100110111 & 5 \\
  101111111101011111111 & 10 \\
  101111111111011111111 & 8 \\
  111111110011111111100 & 7 \\
  111110000101111100001 & 8 \\
  111000011001110000110 & 6 \\
\end{tabular}

\columnbreak

\begin{tabular}{rl}
  000110110010001101100 & 5 \\
  011000000000110000000 & 6 \\
  000000100000000001000 & 8 \\
  000010000010000100000 & 5 \\
  000001111010000011110 & 8 \\
  000111100110001111001 & 3 \\
\end{tabular}

\end{multicols}
The total number of bits in the twenty sequences is 420 and the total number of overlaps is 137.
This yields a $(11,1)$-CS of length $420-137=283$:
\begin{align*}
[&00111011011001110110100011101101011001011010110001010101\\
&10001001011100010010101000100101010011101010100110100101\\
&00110111001001101111111101011111111110111111110011111111\\
&10000101111100001100111000011011001000110110000000001100\\
&00000100000000001000001000010000011110100000111100110001111] \, .
\end{align*}
The following sequences is a $(12,1)$-CS of length $597$:
\begin{align*}
[&101011001110110101100111110010110011111111011001111111110110\\
&111111111000011011111100001010000110000101001111000010100100\\
&001001010010001111001001000111010001100011101001101101110100\\
&110110000110011011000100101101100010001100110001000100110000\\
&100010010001010001001001010101100100101010000010010101000101\\
&010110100010101110010001010110101000101011111100010101111101\\
&111010111110100011011111010010111101101001011100011110101110\\
&001101101111000110111001110011011100101001001110010100000110\\
&010010000011000000000001100000111100110000011110000001101111\\
&000000011111100000001011101010000101110101001110111010100] \, .
\end{align*}
The following sequences is a $(13,1)$-CS of length $1172$:
\begin{align*}
[&10111001111101101110011111110011100111111001111001111110011\\
&01001011110011010001110000110100011101011000000111010110100\\
&10111010110101010000101101010101110011010101011111000000010\\
&11111000010000111110000100010010100001000100010010010001000\\
&10101000001000101010010110001010100101000110101001010000011\\
&11001010000010001101000000100011001110001000110010100110001\\
&10010101101101100101011001110101010110011101101001100111011\\
&00100000111011001001111010110010011111001100100111101111101\\
&00111101111111000111011111110100110111111101111101111111011\\
&01010111111011010111101110110101110000001101011100010111010\\
&11100010000110111000100000001110001000000000011000000000000\\
&11111101010000111111010010001111110100011111111101000101011\\
&11101000101011101010001010111101100001101111011000110011110\\
&11000100101100110001001010010000010010100110100100101001110\\
&11110001001110111100001011101111000110101011110001101100010\\
&10001101100001110001011000011100100110000111001111100001110\\
&01000000001110010001100011100100001010111001000010011011010\\
&00010011011100000100110111011001001101110100011011011101000\\
&10110001101000101100101100001011001011011111110010110111101\\
&000101101111010011001011110100111000010101001110000] \, .
\end{align*}
The following sequence is a $(14,1)$-CS of length $2271$:
\begin{align*}
[&11110111101011011110111101010111110111101010010110111101010\\
&01000011110101001000110001111001000110001010111000110001010\\
&01000011000101001110010001011001110010001010101110010001010\\
&11011001000101010000101000101010000111100101010000110100010\\
&00100011010001000001101010001000001000000001000001000100001\\
&00000100011111100000100011101110000100011101011000100011101\\
&01110010001110101010010001110100001001101110100001000000010\\
&10000100000111111101100000111111100101011111111100101000001\\
&11110010100011001110010100011001000100100011001000000000011\\
&00100000010101100100000010011100100000010000110000000010000\\
&11011000001000011011101111000011011101010101011011101010100\\
&10111110101010010100110101010010100110010110010100110011110\\
&00010011001111001001001001111001001111100111001001111111101\\
&01100111111110111000111111110111111011111110111111100010010\\
&11111110001101001111110001101001011010001101001010111101101\\
&00101011110001000101011110001011101011110001011010111110001\\
&01101011000010101101011000110001101011000110010010111000110\\
&01001001000011001001001100000101001001100001000011001100001\\
&00011110110000100011011111100100011011111010010101011111010\\
&01000010111101001000010010011001000010010100101000010010101\\
&01000001001010100111011011010100111011100110100111011101000\\
&00011101110100011110101110100011100001110100011100111010000\\
&01110011101011100110011101011111000011101011111001111101011\\
&11100110101001111100110101110100000110101110100110000101110\\
&10011010110111010011011111111010011011100111010011011100101\\
&11001101110010111101001110010111100111100010111100111101100\\
&01110011110110110010011110110110010111011010110010111011000\\
&10001011101100001011001101100001011011000101001011011000110\\
&00101101100011111001101100011111011001111011111011001110110\\
&01101100111011010101100111011001011010111011001011011011011\\
&00101101001001100101101000100110101101000100111010101000100\\
&11100011100010011100000010010011100000010101001100000010101\\
&10101001001010110101000101100000101000101100000001010101100\\
&00000101110001000000101110011000111011110011000111110110011\\
&00011110000101100011110000000010011110000000111101110000000\\
&11100110101000011100110110100111100110110100010100110110100\\
&00000011011010000110111111010000110110111110000110110111000\\
&00011011011100001111011011100001100000011100001100000110001\\
&00110000011000010110000011000] \, .
\end{align*}
The following six codewords form a $(11,15,2)$-CSC:
\begin{align*}
&[110011101001001], [111010111001011], [110101110000000], \\
&[010111001000000], [100101000111101], [001010001101100] \, .
\end{align*}
Their extension by ten bits and their associated overlaps are as follows:
\begin{multicols}{3}

\noindent
\begin{tabular}{rl}
  1100111010010011100111010 & 6 \\
  1110101110010111110101110 & 9 \\
\end{tabular}

\columnbreak

\begin{tabular}{rl}
  1101011100000001101011100 & 8 \\
  0101110010000000101110010 & 5 \\
\end{tabular}

\columnbreak

\begin{tabular}{rl}
  1001010001111011001010001 & 9 \\
  0010100011011000010100011 & 2 \\
\end{tabular}

\end{multicols}
The total number of bits in the six sequences is 150 and the total number of overlaps is 39. This yields a $(11,2)$-CS of length $150-39=111$:
\begin{align*}
[&11001110100100111001110101110010111110101110000000110101\\
&110010000 0001011100101000111101100101000110110000101000]
\end{align*}
The following nine codewords form a $(12,13,2)$-CSC:
\begin{align*}
&[1000101000010], [0100000000100], [1000000001101], [0000110101101],[1010110011110],  \\
&[1001110110000], [1000111001001], [1100101111100], [1100101111111] \, .
\end{align*}
Their extension by eleven bits and their associated overlaps are as follows:
\begin{multicols}{3}

\noindent
\begin{tabular}{rl}
  100010100001010001010000 & 6 \\
  010000000010001000000001 & 10 \\
  100000000110110000000011 & 6 \\
\end{tabular}

\columnbreak

\begin{tabular}{rl}
  000011010110100001101011 & 6 \\
  101011001111010101100111 & 6 \\
  100111011000010011101100 & 3 \\
\end{tabular}

\columnbreak

\begin{tabular}{rl}
  100011100100110001110010 & 6 \\
  110010111110011001011111 & 11 \\
  110010111111111001011111 & 1 \\
\end{tabular}
\end{multicols}

The total number of bits in the nine sequences is 216 and the total number of overlaps is 55. This yields a $(12,2)$-CS of length $216-55=161$:
\begin{align*}
[&100010100001010001010000000010001000000001101100000000110101\\
&101000011010110011110101011001110110000100111011000111001001\\
&10001110010111110011001011111111100101111] \, .
\end{align*}
The following sixteen codewords form a $(13,13,2)$-CSC:
\begin{align*}
&[1111111001011], [1001010011010], [0101001101101], [0011011000110], \\
&[0110001100101], [0011001011011], [0101101000001], [1011010000011], \\
&[0001001001111], [0100111010111], [1011111110111], [1111011100000], \\
&[1110000010001], [1000001000000], [0000000111001], [0111000100001] \, .
\end{align*}
Their extension by twelve bits and their associated overlaps are as follows:
\begin{multicols}{3}

\noindent
\begin{tabular}{rl}
  1111111001011111111100101 & 6 \\
  1001010011010100101001101 & 10 \\
  0101001101101010100110110 & 8 \\
  0011011000110001101100011 & 8 \\
  0110001100101011000110010 & 8 \\
  0011001011011001100101101 & 7 \\
\end{tabular}

\columnbreak

\begin{tabular}{rl}
  0101101000001010110100000 & 11 \\
  1011010000011101101000001 & 4 \\
  0001001001111000100100111 & 7 \\
  0100111010111010011101011 & 4 \\
  1011111110111101111111011 & 7 \\
  1111011100000111101110000 & 7 \\
\end{tabular}

\columnbreak

\begin{tabular}{rl}
  1110000010001111000001000 & 10 \\
  1000001000000100000100000 & 5 \\
  0000000111001000000011100 & 6 \\
  0111000100001011100010000 & 0 \\
\end{tabular}

\end{multicols}
The total number of bits in the sixteen sequences is 400 and the total number of overlaps is 108.
This yields a $(13,2)$-CS of length $400-108=292$:
\begin{align*}
[&111111100101111111110010100110101001010011011010101001101100\\
&011000110110001100101011000110010110110011001011010000010101\\
&101000001110110100000100100111100010010011101011101001110101\\
&111111011110111111101110000011110111000001000111100000100000\\
&0100000100000001110010000000111000100001011100010000] \, .
\end{align*}
The following sequence is a $(14,2)$-CS of length $525$:
\begin{align*}
[&001011110100100001011110100111111101110100111111101111100111\\
&111101100011001111101100011011011000100011011011000011111011\\
&011000011100011011000011100100110100011100100111001100010100\\
&111001100101011001001100101011101111110101011101111000000101\\
&101111000000010001111000000011001101001000011001101010010011\\
&001101010010010110101010010010110101011110010110101011001101\\
&011101011001101011110111101101011110111010010111100111010010\\
&111011100000010111011100101000101011100101000100000011101000\\
&100000000110000100000000110010001010000110010] \, .
\end{align*}
The following sequence is a $(15,2)$-CS of length $907$:
\begin{align*}
[&000010111100110000001011110011001000010111001100100000101110\\
&110010000010100110001000001010010000000000101001000011101000\\
&000100001110100101001000111010010100111011101111010011101110\\
&111010011110111011101001000001110110100100000110010001010000\\
&011001001010000001100100101010110011110010101011001011110111\\
&101100101111010110001111111101011000110100010101100011010011\\
&100101101101001110010011110001111001001111000010010100111100\\
&001001101111110000100110111110100011111011111010001011011101\\
&101000101101110011110011000111001111001101110011111100110111\\
&000100010011011100010100101101010001010010110011010001001011\\
&001101101001101100110110100001100011011010000110111001011100\\
&011011100101100000011110010110000011100001111000001110000011\\
&001111111000001100110000100000110011000101100011001100010101\\
&111100110001010100111010000101010011101011100101001110101111\\
&111011111010111111101101011101010100110101110101011000000111\\
&0101011] \, .
\end{align*}
The following five codewords form a $(13,13,3)$-CSC:
\begin{align*}
&[0110111110111], [1111101100010], [0110001101000], [1101000001001], [0100000100000] \, .
\end{align*}
Their extension by twelve bits and their associated overlaps are as follows:
\begin{multicols}{3}

\noindent
\begin{tabular}{rl}
  0110111110111011011111011 & 8 \\
  1111101100010111110110001 & 7 \\
\end{tabular}

\columnbreak

\begin{tabular}{rl}
  0110001101000011000110100 & 6 \\
  1101000001001110100000100 & 10 \\
\end{tabular}

\columnbreak

\begin{tabular}{rl}
  0100000100000010000010000 & 1 \\
\end{tabular}

\end{multicols}

The total number of bits in the five sequences is 125 and the total number of overlaps is 32.
This yields a $(13,3)$-CS of length $125-32=93$:
\begin{align*}
[&011011111011101101111101100010111110110001101000011000110100\\
&000100111010000010000001000001000] \, .
\end{align*}
The following ten codewords form a $(14,15,3)$-CSC:
\begin{align*}
&[110011000000010], [011010101001111], [010011101011110], [110101111000100], \\
&[000101100001001], [101101011111101], [111111110001101], [000111011001010], \\
&[001001000101000], [010100001100010] \, .
\end{align*}
Their extension by thirteen bits and their associated overlaps are as follows:
\begin{multicols}{3}

\noindent
\begin{tabular}{rl}
  1100110000000101100110000000 & 1 \\
  0110101010011110110101010011 & 6 \\
  0100111010111100100111010111 & 8 \\
  1101011110001001101011110001 & 4 \\
\end{tabular}

\columnbreak

\begin{tabular}{rl}
  0001011000010010001011000010 & 2 \\
  1011010111111011011010111111 & 6 \\
  1111111100011011111111100011 & 5 \\
  0001110110010100001110110010 & 4 \\
\end{tabular}

\columnbreak

\begin{tabular}{rl}
  0010010001010000010010001010 & 5 \\
  0101000011000100101000011000 & 0 \\
\end{tabular}

\end{multicols}
The total number of bits in the ten sequences is 280 and the total number of overlaps is 41.
This yields a $(14,3)$-CS of length $280-41=239$:
\begin{align*}
[&111001001110101111000100110101111000101100001001000101100001\\
&011010111111011011010111111110001101111111110001110110010100\\
&00111011001001000101000001001000101000011000100101000011000] \,.
\end{align*}
The following sequence is a $(15,3)$-CS of length $406$:
\begin{align*}
[&100000010100101100000010100100000100010100100000100000010110\\
&010100000010110010001111010110010001111110001110101111110001\\
&110111110101111110111110101111100011001101111100011000110111\\
&100011000110110001011101110110001011101101111001111101101111\\
&001101010000011001101010000011100110010000011100110111000010\\
&100110111000010100000000000010100000001110100100000001110101\\
&1010011111101011010011111001100101011111001100] \, .
\end{align*}
\end{scriptsize}

\newpage

\section{$(15,1)$-CS from the Hamming Code of Length 15}
\label{sec:Ham15}

\begin{scriptsize}
\begin{multicols}{3}
\begin{align*}
& 00001011101010100001011101010 & 1 \\
& 00010011001110100010011001110 & 1 \\
& 00001010001110100001010001110 & 7 \\
& 00011100111111100011100111111 & 8 \\
& 00111111111010100111111111010 & 1 \\
& 00010011100101100010011100101 & 5 \\
& 00101101101110100101101101110 & 1 \\
& 00000111100001100000111100001 & 6 \\
& 1000010000100001000 & 8 \\
& 00001000100010100001000100010 & 9 \\
& 00010001001001100010001001001 & 9 \\
& 00100100111101100100100111101 & 8 \\
& 00111101010110100111101010110 & 1 \\
& 00011001111110100011001111110 & 1 \\
& 00011011010010100011011010010 & 4 \\
& 00101011111110100101011111110 & 1 \\
& 00001101111010100001101111010 & 1 \\
& 00001110011001100001110011001 & 7 \\
& 00110011011010100110011011010 & 1 \\
& 00000110101110100000110101110 & 1 \\
& 00001110110010100001110110010 & 4 \\
& 00101110100110100101110100110 & 1 \\
& 00000101100110100000101100110 & 1 \\
& 00000011110110100000011110110 & 1 \\
& 00000010010010100000010010010 & 8 \\
& 10010010010010010 & 7 \\
& 00100101011001100100101011001 & 3 \\
& 00100101110010100100101110010 & 1 \\
& 00000001011010100000001011010 & 10 \\
& 00010110100100100010110100100 & 5 \\
& 00100111011110100100111011110 & 1 \\
& 00011100100110100011100100110 & 8 \\
& 00100110111010100100110111010 & 1 \\
& 00011010110110100011010110110 & 1 \\
& 00001100101100100001100101100 & 2 \\
& 00010101101100100010101101100 & 2 \\
& 00001010111100100001010111100 & 2 \\
& 00011111011100100011111011100 & 2 \\
& 00011001001100100011001001100 & 2 \\
& 00010011111100100010011111100 & 2 \\
& 00000110011100100000110011100 & 2 \\
& 00000000100111100000000100111 & 10 \\
& 00001001110100100001001110100 & 2 \\
& 00000000010101100000000010101 & 10 \\
& 00000101010100100000101010100 & 2 \\
& 00001101001000100001101001000 & 3 \\
& 00001011011000100001011011000 & 3 \\
& 00000111111000100000111111000 & 3
\end{align*}

\begin{align*}
& 00000001101000100000001101000 & 3 \\
& 00000100110000100000100110000 & 4 \\
& 00000010100000100000010100000 & 5 \\
& 000000000000000 & 10 \\
& 00000000001100100000000001100 & 10 \\
& 00000011000100100000011000100 & 6 \\
& 00010001100010100010001100010 & 5 \\
& 00010101000111100010101000111 & 6 \\
& 00011111101110100011111101110 & 1 \\
& 00010100111010100010100111010 & 8 \\
& 00111010111011100111010111011 & 0 \\
& 00010010101010100010010101010 & 9 \\
& 01010101011101101010101011101 & 0 \\
& 00010010110011100010010110011 & 0 \\
& 00001111001111100001111001111 & 6 \\
& 00111101001111100111101001111 & 6 \\
& 00111110110101100111110110101 & 0 \\
& 00001011110011100001011110011 & 0 \\
& 00001100011110100001100011110 & 8 \\
& 00011110010011100011110010011 & 7 \\
& 00100111110101100100111110101 & 0 \\
& 00010110111101100010110111101 & 0 \\
& 00001001101101100001001101101 & 0 \\
& 00001001011111100001001011111 & 9 \\
& 00101111101001100101111101001 & 0 \\
& 00001001000110100001001000110 & 6 \\
& 00011010011101100011010011101 & 7 \\
& 00111011011111100111011011111 & 0 \\
& 00001000111011100001000111011 & 9 \\
& 00011101101001100011101101001 & 0 \\
& 00010001111011100010001111011 & 9 \\
& 00111101111101100111101111101 & 0 \\
& 00010101110101100010101110101 & 0 \\
& 00001100110101100001100110101 & 8 \\
& 00110101010011100110101010011 & 0 \\
& 00000111010011100000111010011 & 0 \\
& 00010101011110100010101011110 & 8 \\
& 01011110111011101011110111011 & 0 \\
& 00000110110111100000110110111 & 0 \\
& 00010111011001100010111011001 & 0 \\
& 00010011010111100010011010111 & 6 \\
& 01011111011111101011111011111 & 0 \\
& 00001010100101100001010100101 & 8 \\
& 1010010100101001010 & 6 \\
& 00101011010101100101011010101 & 0 \\
& 00000101111111100000101111111 & 0 \\
& 00001110101011100001110101011 & 0 \\
& 00000110000101100000110000101 & 7
\end{align*}

\begin{align*}
& 00001010010111100001010010111 & 7 \\
& 00101111011011100101111011011 & 8 \\
& 11011011011011011 & 0 \\
& 00000101001101100000101001101 & 6 \\
& 00110110011011100110110011011 & 0 \\
& 00010111110010100010111110010 & 4 \\
& 00101011100111100101011100111 & 8 \\
& 1110011100111001110 & 6 \\
& 00111011101101100111011101101 & 0 \\
& 00000100101001100000100101001 & 8 \\
& 00101001111001100101001111001 & 0 \\
& 00000100011011100000100011011 & 8 \\
& 00011011001011100011011001011 & 6 \\
& 00101110111111100101110111111 & 0 \\
& 00011001100111100011001100111 & 0 \\
& 00000011101111100000011101111 & 0 \\
& 00000011011101100000011011101 & 0 \\
& 00000111001010100000111001010 & 6 \\
& 00101010011010100101010011010 & 7 \\
& 00110100110111100110100110111 & 8 \\
& 00110111111111100110111111111 & 9 \\
& 111111111111111 & 0 \\
& 00000010111001100000010111001 & 0 \\
& 00001101010001100001101010001 & 4 \\
& 00010111101011100010111101011 & 0 \\
& 00011101011011100011101011011 & 0 \\
& 00000010001011100000010001011 & 7 \\
& 00010110001111100010110001111 & 7 \\
& 00011111110111100011111110111 & 8 \\
& 1111011110111101111 & 0 \\
& 00000001110001100000001110001 & 4 \\
& 00010100100011100010100100011 & 5 \\
& 00011010101111100011010101111 & 7 \\
& 01011111101101101011111101101 & 0 \\
& 00001111010110100001111010110 & 8 \\
& 1101011010110101101 & 0 \\
& 00010110010110100010110010110 & 7 \\
& 00101101110111100101101110111 & 0 \\
& 00000001000011100000001000011 & 6 \\
& 00001101100011100001101100011 & 7 \\
& 1100011000110001100 & 7 \\
& 00011001010101100011001010101 & 8 \\
& 01010101101111101010101101111 & 0 \\
& 00000000111110100000000111110 & 10 \\
& 00001111100100100001111100100 & 5 \\
& 00100101101011100100101101011 & 0 \\
& 00011011111001100011011111001 & 0 \\
& 00001111111101100001111111101 & 0
\end{align*}
\end{multicols}

The total number of bits in the one hundred forty-four sequences is 4064 and the total number of overlaps is 548.
This yields a $(15,1)$-CS of length $4064-548=3516$.
\end{scriptsize}

\section{$(16,1)$-CS from Self-Dual Sequences of Length 32}
\label{sec:SD16}

\begin{scriptsize}
\begin{align*}
& 1110010011010011000110110010110011100100110100010001101100101110111001001101001 & 11 \\
& 0100110100100010101100101101110101001101001001101011001011011001010011010010001 & 11 \\
& 1101001000110110001011001100100111010011001101100010110111001001110100100011011 & 9 \\
& 1000110110100101011100100101101010000101101001010111101001011010100011011010010 & 9 \\
& 0110100101000001100111101011111001100001010000011001011010111110011010010100000 & 9 \\
& 0101000000101100101011111111001101010000000011001010111111010011010100000010110 & 8 \\
& 0001011010000010111010010111110100010010100000101110110101111101000101101000001 & 7 \\
& 1000001111000101001111000011101011000011110001010111110000111010100000111100010 & 5 \\
& 0001011110000111111010000111100000010011100001111110110001111000000101111000011 & 10 \\
& 1111000011011000000001110010011111111000110110000000111100100111111100001101100 & 8 \\
& 0110110010001000100100110111011101101101100010001001001001110111011011001000100 & 8 \\
& 0100010000101111101110111101000001000100001010111011101111010100010001000010111 & 12 \\
& 0010000101110000110011101000111100110001011100001101111010001111001000010111000 & 13 \\
& 1000010111000100011010100011101110010101110001000111101000111011100001011100010 & 12 \\
& 0010111000101000110100011101011101101110001010001001000111010111001011100010100 & 10 \\
& 1100010100100001001110111101111011000100001000010011101011011110110001010010000 & 8 \\
& 1001000010001010011011110111010010010000100010110110111101110101100100001000101 & 7 \\
& 1000101110111100011101000100000110001011101111100111010001000011100010111011110 & 8 \\
& 1101111000100111001000011101100011011110011001110010000110011000110111100010011 & 7 \\
& 0010011001100000110110011001111101100110011000001001100110011111001001100110000 & 10 \\
& 1100110000110101001100111100101011001100011101010011001110001010110011000011010 & 9 \\
& 0000110101111111111100101000000000001101111111111111001000000000000011010111111 & 10 \\
& 1010111111000000010100000011101110101111110001000101000000111111101011111100000 & 10 \\
& 1111100000110000000001011100111111111010001100000000011111001111111110000011000 & 10 \\
& 0000011000110001111110011100111000000110001100111111100111001100000001100011000 & 11 \\
& 0110001100001000100111001111001101100011000011001001110011110111011000110000100 & 9 \\
& 1100001001101100001111011001001111000010011111000011110110000011110000100110110 & 10 \\
& 0100110110110100101100100100101101001101111101001011001000001011010011011011010 & 9 \\
& 0110110101001101100100101011001001100101010011011001101010110010011011010100110 & 7 \\
& 0100110000001101101100111111001001011100000011011010001111110010010011000000110 & 12 \\
& 0110000001100110100111111001100100100000011001101101111110011001011000000110011 & 9 \\
& 0001100110011010111001100110010100111001100110101100011001100101000110011001101 & 6 \\
& 0011010001110010110010111000110101110100011100101000101110001101001101000111001 & 11 \\
& 0100011100111011101110001100010001000111000110111011100011100100010001110011101 & 10 \\
& 1110011101011110000110001110000111100111000111100001100010100001111001110101111 & 9 \\
& 1101011110101011001010000111010011010111100010110010100001010100110101111010101 & 7 \\
& 1010101011111011010101010000010010111010111110110100010100000100101010101111101 & 10 \\
& 0101111101010010101000001010110111011111010100100010000010101101010111110101001 & 8 \\
& 1010100100111001010100101100011010101101001110010101011011000110101010010011100 & 10 \\
& 0010011100010010110110001110111100100111000100001101100011101101001001110001001 & 11 \\
& 0111000100100000100111101101111101100001001000001000111011011111011100010010000 & 11 \\
& 0001001000010100101011011110101101010010000101001110110111101011000100100001010 & 9 \\
& 1000010101100000011110111001111110000100011000000111101010011111100001010110000 & 11 \\
& 0101011000001010101010011101010101010110001010101010100111110101010101100000101 & 10 \\
& 1100000101000010000111101011110111100001010000100011111010111101110000010100001 & 5 \\
& 0000101001011011111101011010010000011010010110111110010110100100000010100101101 & 13 \\
& 0010100101101000110101101001011100111001011010001100011010010111001010010110100 & 8 \\
& 1011010000100100010011111101101110110000001001000100101111011011101101000010010 & 13 \\
& 1101000010010110011011110110100110010000100101100010111101101001110100001001011 & 5 \\
& 0101101111001111101001000011000001011111110011111010000000110000010110111100111 & 10 \\
& 0111100111101101100000100001001001111101111011011000011000010010011110011110110 & 9 \\
& 0111101100101010100001001101010101101011001010101001010011010101011110110010101 & 13 \\
& 1110110010101101000100110101001011111100101011010000001101010010111011001010110 & 10 \\
& 1001010110001111011010100111000110010101100011100110101001110000100101011000111 & 7 \\
& 1000111101111101011100001000001010001111111111010111000000000010100011110111110 & 9 \\
& 1101111100010111001000001110100011011101000101110010001011101000110111110001011 & 12 \\
& 1111100010111110000001110100000110111000101111100100011101000001111110001011111 & 11 \\
& 1000101111110111011101000000100010001011110101110111010000101000100010111111011 & 8 \\
& 1111101101101111000001001001000011111111011011110000000010010000111110110110111 & 10 \\
& 0110110111010011100100100011110001101101110000111001001000101100011011011101001 & 9 \\
& 0111010010101011100010110101110001110100101000111000101101010100011101001010101 & 8 \\
& 0101010110011111101010100110000001010111100111111010100001100000010101011001111 & 12 \\
& 1010110011110000010100110000111110101100110100000101001100101111101011001111000 & 11 \\
& 1100111100010000001100001110011111001111000110000011000011101111110011110001000 & 0
\end{align*}
The total number of bits in the sixty-four sequences is 5056 and the total number of overlaps is 594. This yields a $(16,1)$-CS of length $5056-594=4462$.
\end{scriptsize}


\section*{Acknowledgement}
Parts of the paper were presented in IEEE International Symposium on Information Theory
and appeared in {\em Proceedings IEEE Symposium on Information Theory,} Athens, Greece 2024, pp. 1343--1348.

\end{document}